\documentclass[10pt,reqno]{amsart}

\usepackage[dvips]{graphicx}
\usepackage{enumerate}
\usepackage{amsmath,amssymb,amsthm}
\usepackage[usenames]{color}

\theoremstyle{plain}
\newtheorem{thm}{Theorem}[section]
\newtheorem{cor}[thm]{Corollary}
\newtheorem{lem}[thm]{Lemma}
\newtheorem{prop}[thm]{Proposition}
\theoremstyle{definition}
\newtheorem{defn}[thm]{Definition}
\newtheorem{rem}[thm]{Remark}
\newtheorem{rems}[thm]{Remarks}

\def\ZZ{\mathbb Z}

\begin{document}
\title{The  continued fractions ladder of $(\sqrt[3]{m},\sqrt[3]{m^2})$}
\author{Mitja Lakner, Peter Petek, Marjeta \v Skapin Rugelj  }
\thanks{ M. Lakner,  M. \v{S}kapin Rugelj: University of Ljubljana, Faculty of Civil and Geodetic Engineering,
Jamova 2, 1000 Ljubljana, Slovenia.\\
P. Petek: University of Ljubljana, Faculty of Education,
Kardeljeva plo\v{s}\v{c}ad 16, 1000 Ljubljana, Slovenia.\\ \textit{ E-mails:}
mlakner@fgg.uni-lj.si, Peter.Petek@guest.arnes.si,
mskapin@fgg.uni-lj.si.}
\date{\today}

\maketitle

\begin{abstract}
Quadratic irrationals posses a periodic continued fraction expansion. Much
less is known about cubic irrationals.
We do not even know if the partial
quotients are bounded, even though extensive computations suggest they might
follow Kuzmin's probability law.
Results are given for sequences of
partial quotients of
$\sqrt[3]{m}$ and $\sqrt[3]{m^2}$ with $m$ noncube. A big partial quotient in one
sequence finds a connection in the other.
\end{abstract}

\section{Introduction}

Several authors have considered simultaneous rational approximations
to pairs of irrationals
$(\alpha, \beta)$  \cite{Be, Ch, WWD} and new generalized concepts were
developed \cite{BrP}. Here however we observe a different
parallelism for the specific pair $(\sqrt[3]{m}, \sqrt[3]{m^2})$ and the usual continued
fraction expansion.
Our long term goal is the proof of the\\[2mm]
 \noindent{\bf Hypothesis}. The partial quotients in the continued fraction
expansion of $\sqrt[3]{2}$ are unbounded.\\

Already in \cite{H} this question is asked for algebraic numbers of degrees
higher than 2.

A starting point for our present observation could be the long tables of partial
quotients for $(\sqrt[3]{2},\sqrt[3]{4})$ (among other algebraic irrationals) in \cite{LT},
where  the bigger ones are singled out.
In Figure \ref{fig1} we notice parallel apparition of big partial quotients in both the
continued fraction expansion sequences.

\begin{figure}[!htbp]
\centering
\includegraphics[width=130mm]{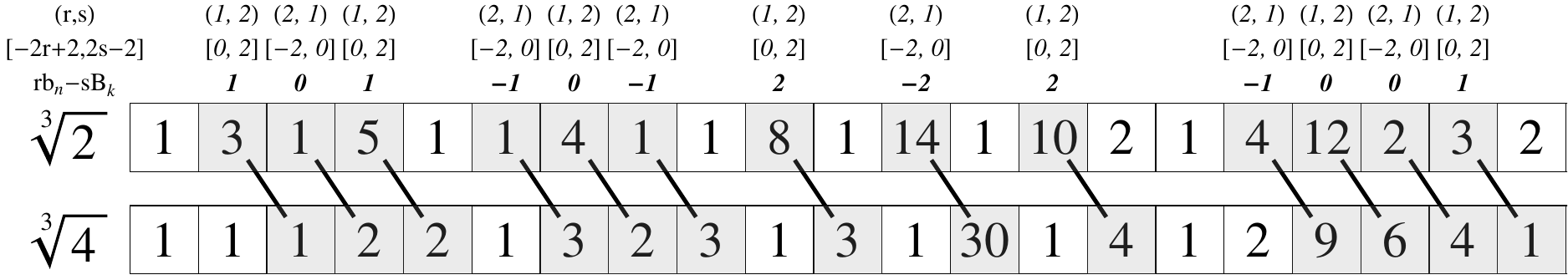}
\caption{Continued fraction ladder of  $(\sqrt[3]{2},\sqrt[3]{4})$}\label{fig1}
\end{figure}

A relatively big partial quotient in one sequence is connected to
roughly half that quotient in the other sequence. These relations
are formalized and analyzed in the sequel. "Big" here can be as small as $2$.

\begin{rems}
Numerous numerical experiments have been carried out to support
the Kuzmin statistics, giving the probability of a certain partial quotient
at $P(b_n=k)=\log_2{(k+1)^2\over k(k+2)}$ \cite{RDM, LT, Bru}.

It is also of interest that cubic irrationals appear in physics in studying
chaotic and quasiperiodic motions \cite{MH, CSG}.
\end{rems}

Let $\xi$ be any real number and we start the continued fraction process by
saying $\xi_0=\xi$, taking integer part $b_0=\lfloor \xi\rfloor$ and setting
$$\xi_1={1\over \xi_0-b_0},\;\; b_1=\lfloor\xi_1\rfloor$$
only to continue as long as
we can in the same fashion
\begin{equation}\label{ksi}
\xi_n={1\over \xi_{n-1}-b_{n-1}},\;\; b_n=\lfloor\xi_n\rfloor .
\end{equation}
The process eventually stops for a rational $\xi={a\over b}$
and continues indefinitely for irrational one.

We reap the approximations, convergents ${p_n\over q_n}$ to $\xi$, defined
by
$$p_{-1}=1,\; q_{-1}=0,\; p_0=b_0,\;q_0=1$$
$$p_n= b_np_{n-1}+p_{n-2},\;\; q_n=b_nq_{n-1}+q_{n-2}$$

We shall need some elementary results from \cite{P} or \cite {H}
and we quote them here.

\begin{itemize}
\item Let us have an irrational $\xi=\xi_0$ and the following complete
quotients are denoted by $\xi_n$, the corresponding convergents by
${p_n\over q_n}$, then we can express:
\begin{equation}\label{copml_quot}
\xi_n=-{p_{n-2}-q_{n-2}\xi \over p_{n-1}-q_{n-1}\xi }.
\end{equation}
\item Even convergents are smaller than the irrational and odd ones are
greater
\begin{equation}\label{sod_lih}
{p_{2j}\over q_{2j}}<\xi<{p_{2j+1}\over q_{2j+1}}
\end{equation}
 and
\begin{equation}\label{pqpq}
p_n q_{n-1}-p_{n-1}q_n=(-1)^{n-1}
\end{equation}
yielding the fact, that the
convergents are in their lowest terms.
\item ${p_n\over q_n}=\xi+(-1)^{n-1}{|\delta|\over q_n^2}$ with $|\delta|<1$
and if $b=b_{n+1}$ be the next partial quotient there is an
estimate ${1\over b+2}<|\delta|<{1\over b}$.
\item For any rational ${p\over q}$ there is a sufficient condition $|\delta|<{1\over 2}$
to be one of the convergents where $\delta=(\frac{p}{q}-\xi)q^2$.
\item Let $\frac{p_n}{q_n}$ and $\frac{p_N}{q_N}$ be two convergents of $\xi$. Then the following
statements are equivalent
\begin{itemize}
\item[$\ast$] $n < N$
\item[$\ast$] $p_n < p_N$
\item[$\ast$] $q_n < Q_N$
\item[$\ast$] $|\frac{p_n}{q_n}-\xi|>|\frac{p_N}{q_N}-\xi|$
\end{itemize}
The last inequality is a consequence of (\ref{copml_quot}).
\item The sequence of "relative errors" $|1- \frac{\xi}{p_n/q_n}|$ is decreasing.
\begin{eqnarray*}
1< |\xi_n|\stackrel{(\ref{copml_quot})}{=} \left|-{p_{n-2}-q_{n-2}\xi \over p_{n-1}-q_{n-1}\xi } \right|
& = & \frac{ \left|1- \frac{\xi}{p_{n-2}/q_{n-2}}\right| p_{n-2}}{ \left|1- \frac{\xi}{p_{n-1}/q_{n-1}}\right| p_{n-1}}\\
& < & \frac{ \left|1- \frac{\xi}{p_{n-2}/q_{n-2}}\right|}{ \left|1- \frac{\xi}{p_{n-1}/q_{n-1}}\right|}
\end{eqnarray*}

\end{itemize}
Here all fractions are in their lowest terms.

\section{Relations in the sequences of convergents}

Let $m$ be a positive integer, noncube. Then we can express the cubic
roots $\xi=\sqrt[3]{m}$  and   $\eta=\sqrt[3]{m^2}$ as infinite continued fractions.
The triplets
$({p_{n-1}\over q_{n-1}}, \xi_n, b_n)$ and $({P_{k-1}\over
Q_{k-1}},\eta_k,  B_k)$ denotes convergents, complete quotients and partial quotients
respectively.

\begin{defn}\label{connected} We say that the triplets
$({p_{n-1}\over q_{n-1}}, \xi_n, b_n)$ and $({P_{k-1}\over
Q_{k-1}},\eta_k,  B_k)$ are \emph{connected}, if
$$
{p_{n-1}\over q_{n-1}} \cdot {P_{k-1}\over Q_{k-1}}=m.
$$
We call the triplets together with connections \emph{ladder} of $(\sqrt[3]{m},\sqrt[3]{m^2})$ (see Figure \ref{fig1} and \ref{fig2}).

\end{defn}

We will need the following identities
\begin{lem}\label{zveza}
\begin{eqnarray*}
 m \frac{q}{p}-\sqrt[3]{m^2} &=& - \frac{q}{p} \sqrt[3]{m^2} \left(  \frac{p}{q}-\sqrt[3]{m} \right) \\
\left( m \frac{q}{p}-\sqrt[3]{m^2} \right) p^2 &=& - \frac{p}{q} \sqrt[3]{m^2} \left(  \frac{p}{q}-\sqrt[3]{m} \right) q^2
\end{eqnarray*}
\end{lem}

\begin{prop}
Two different connections do not intersect.
\end{prop}

\begin{proof}
Let take two connections with convergents $\frac{p}{q}$, $\frac{p'}{q'}$ to $\sqrt[3]{m}$ such that  $|\frac{p}{q}-\sqrt[3]{m}|>|\frac{p'}{q'}-\sqrt[3]{m}|$.
We want to prove that $|m\frac{q}{p}-\sqrt[3]{m^2}|>|m\frac{q'}{p'}-\sqrt[3]{m^2}|$.

Using Lemma \ref{zveza} and decreasing property of relative errors we get
$$\frac{\left| m \frac{q}{p}-\sqrt[3]{m^2} \right|}{\left| m \frac{q'}{p'}-\sqrt[3]{m^2} \right|}
=
\frac{\frac{q}{p} \sqrt[3]{m^2} \left|  \frac{p}{q}-\sqrt[3]{m} \right| }
     {\frac{q'}{p'} \sqrt[3]{m^2} \left|  \frac{p'}{q'}-\sqrt[3]{m} \right| }
= \frac{\left| 1- \frac{\sqrt[3]{m}}{p/q} \right|}{\left| 1- \frac{\sqrt[3]{m}}{p'/q'} \right|}
>1.
$$
\end{proof}

\begin{prop}
Let $\frac{p}{q}$ be convergent to $\sqrt[3]{m}$ or $\sqrt[3]{m^2}$ with partial
quotient $b\ge 2m+1$. Then we have the connection to $m \frac{q}{p}$.
\end{prop}

\begin{proof}
Let $\frac{p}{q}$ be convergent to $\sqrt[3]{m}$.
We have to see that $m \frac{q}{p}$ is convergent to
$\sqrt[3]{m^2}$. Let $\frac{q_1}{p_1}$ be $m \frac{q}{p}$ in its
lowest terms.

\begin{eqnarray*}
\left|\left(\frac{q_1}{p_1}-\sqrt[3]{m^2}\right)p_1^2\right|&\le&
\left|\left(m\frac{q}{p}-\sqrt[3]{m^2}\right)p^2\right|\le
\frac{p}{q} \sqrt[3]{m^2} \left|\frac{p}{q}-\sqrt[3]{m}
\right|q^2\\
&\le& \left(\sqrt[3]{m}+\frac{1}{b q^2}\right)
\sqrt[3]{m^2}\frac{1}{b}< \left(m +\frac{m}{2m
1^2}\right)\frac{1}{2m+1}=\frac{1}{2}
\end{eqnarray*}
We have used Lemma \ref{zveza} and the fact that $\left|
\frac{p}{q}-\sqrt[3]{m} \right| < \frac{1}{b q^2}$.

Proof for $\sqrt[3]{m^2}$ is analogous.
\end{proof}

\begin{lem}\label{parity}
If the triplets
$({p_{n-1}\over q_{n-1}}, \xi_n, b_n)$ and $({P_{k-1}\over
Q_{k-1}},\eta_k,  B_k)$ are connected, then $n$ and $k$ are of different parity: $(-1)^{k-1}=(-1)^n$.
\end{lem}

\begin{proof}
It is enough to prove for $n$ is odd. From (\ref{sod_lih})
it follows that
$${p_{n-1}\over q_{n-1}}{P_{k-1}\over Q_{k-1}}=\sqrt[3]{m} \, \sqrt[3]{m^2} >  {p_{n-1}\over q_{n-1}} \sqrt[3]{m^2}.$$
So we have $\sqrt[3]{m^2} < {P_{k-1}\over Q_{k-1}}$ and $k$ is even.

\end{proof}

\begin{thm} \label{izrek}
Let  $({p_{n-1}\over q_{n-1}}, \xi_n, b_n)$
and $({P_{k-1}\over Q_{k-1}}, \eta_k, B_k)$ be connected. Then
there exist natural numbers $r$, $s$, such that $rs=m$ and
\begin{equation}\label{glavno} -2r+2\le rb_n-sB_k\le 2s-2.
\end{equation}
\end{thm}

\begin{proof}
Because convergents ${p_{n-1}\over q_{n-1}}$, ${P_{k-1}\over
Q_{k-1}}$ are reduced, it follows from Definition \ref{connected}
that
\begin{equation}\label{rs}
r:={p_{n-1}\over Q_{k-1}},\qquad s:={P_{k-1}\over q_{n-1}}
\end{equation}
are natural numbers.

Using (\ref{copml_quot}) and (\ref{rs}) we get
\begin{eqnarray*}
r\xi_n-s\eta_k&=&-r{p_{n-2}-q_{n-2}\sqrt[3]{m} \over p_{n-1}-q_{n-1}\sqrt[3]{m}
}+s{P_{k-2}-Q_{k-2}\sqrt[3]{m^2} \over P_{k-1}-Q_{k-1}\sqrt[3]{m^2} }\\
&=& {-r p_{n-2}+r q_{n-2}\sqrt[3]{m} \over r Q_{k-1}- q_{n-1}\sqrt[3]{m}}+{s
P_{k-2}-s Q_{k-2}\sqrt[3]{m^2} \over s q_{n-1}-Q_{k-1}\sqrt[3]{m^2}}
\end{eqnarray*}

To obtain the same denominator, the second fraction is extended
with $\sqrt[3]{m}$ and the factor $s$ divided out
\begin{eqnarray}\label{dd}\nonumber
r\xi_n-s\eta_k&=&{-r p_{n-2}+rq_{n-2}\sqrt[3]{m} \over r Q_{k-1}-
q_{n-1}\sqrt[3]{m}}+{P_{k-2}\sqrt[3]{m} -m Q_{k-2}\over q_{n-1}\sqrt[3]{m} -r
Q_{k-1}}\\ &=&{r(sQ_{k-2}-p_{n-2})-(P_{k-2}-rq_{n-2})\sqrt[3]{m} \over
rQ_{k-1}- q_{n-1}\sqrt[3]{m}}
\end{eqnarray}

Let us  solve the linear diophantine equation
\begin{equation}\label{diof_en}
p_{n-1}x-q_{n-1}y=(-1)^n r.
\end{equation}
From (\ref{pqpq}) it follows that
$p_{n-1} q_{n-2}-q_{n-1}p_{n-2}=(-1)^{n}$.
Hence  one of the solution of (\ref{diof_en}) is
$(x_0,y_0)=(rq_{n-2}, rp_{n-2})$.

Let us prove that
$(x_1, y_1)=(P_{k-2}, mQ_{k-2})$ is also the solution of (\ref{diof_en})
using (\ref{rs}), (\ref{pqpq}) and Lemma \ref{parity} we get
\begin{eqnarray*}
p_{n-1}x_1-q_{n-1}y_1&=& p_{n-1} P_{k-2}-q_{n-1} m Q_{k-2}\\
                     &=& r Q_{k-1} P_{k-2} - \frac{P_{k-1}}{s} r s  Q_{k-2}\\
                     &=& -r (P_{k-1} Q_{k-2} -Q_{k-1} P_{k-2})\\
                     &=& (-1)^{k-1} r\\
                     &=& (-1)^n r\\
\end{eqnarray*}
Since $p_{n-1}$ and $q_{n-1}$ are coprime, the
general solution of (\ref{diof_en}) can be written as
$$(x,y)=(x_0,y_0)+t(q_{n-1},p_{n-1}), \quad t\in \ZZ.$$
Hence
$$(P_{k-2}, mQ_{k-2})=(rq_{n-2}, rp_{n-2})+t(q_{n-1},p_{n-1})$$
and the parameter $t$ can be expressed in two forms
\begin{equation}\label{t}
t={P_{k-2}-rq_{n-2}\over q_{n-1}}={mQ_{k-2}-rp_{n-2}\over p_{n-1}} \in \ZZ.
\end{equation}

Using these two forms in the equation (\ref{dd}) it follows that
$r\xi_n-s\eta_k=t$. On the other hand from (\ref{t}) and (\ref{rs}) we can estimate
$$t={P_{k-2}-rq_{n-2}\over q_{n-1}} < {P_{k-1}-rq_{n-2}\over q_{n-1}}= {sq_{n-1}-rq_{n-2}\over q_{n-1}}<s$$
and
$$t={mQ_{k-2}-rp_{n-2}\over p_{n-1}} > {mQ_{k-2}-rp_{n-1}\over p_{n-1}}= {r s Q_{k-2}-r^2 Q_{k-1}\over r Q_{k-1}}>-r.$$
Since $t\in \ZZ$ it follows
\begin{equation}\label{neenacba}
-r+1\le r\xi_n-s\eta_k \le s-1.
\end{equation}
We decompose the complete quotients into their integral and
fractional parts:
$$\xi_n=\lfloor\xi_n\rfloor + (\xi_n),\qquad {i\over r}<(\xi_n)<{i+1\over r},\qquad 0\le i\le r-1$$
$$\eta_k=\lfloor\eta_k\rfloor + (\eta_k),\qquad {j\over s}<(\eta_k)<{j+1\over s},\qquad 0\le j\le s-1$$
If we use this decomposition in the equation (\ref{neenacba}), we
get
$$
-r+1\le r b_n +i + \epsilon_1 - s B_k -j - \epsilon_2\le s-1
$$
$$
-2r+2+\epsilon_2-\epsilon_1\le r b_n - s B_k \le
2s-2+\epsilon_2-\epsilon_1
$$
where $0<\epsilon_1, \epsilon_2<1$ and $\epsilon_2-\epsilon_1$ is
between -1 and 1. Because all other numbers are integers, we
finaly get
$$-2r+2\le rb_n-sB_k\le 2s-2.$$
\end{proof}

\begin{figure}[!htbp]
\centering
\includegraphics[width=130mm]{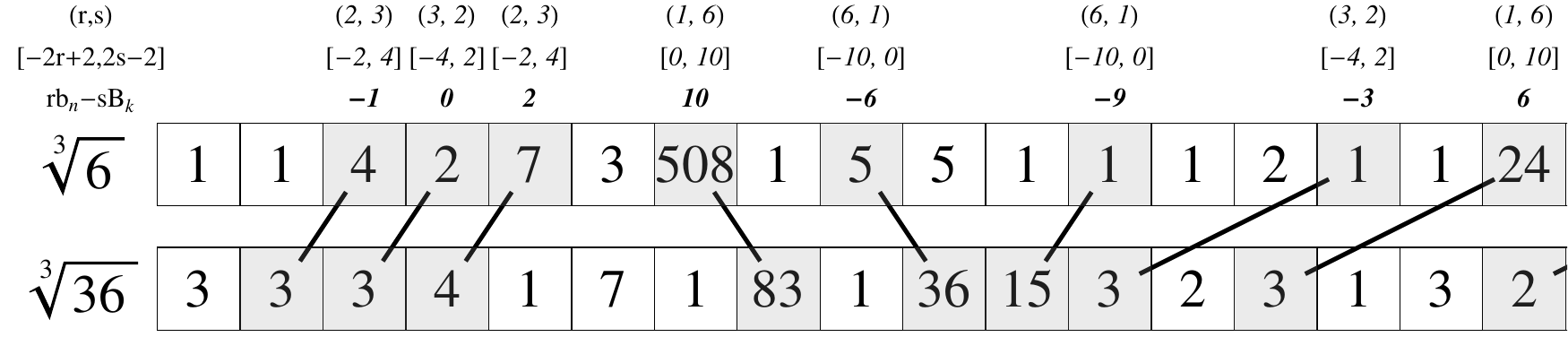}
\caption{Continued fraction ladder of $(\sqrt[3]{6},\sqrt[3]{36})$}\label{fig2}
\end{figure}

\begin{cor}
If in the above Theorem \ref{izrek} prime  $m$ is taken, then the ratio between the
connected partial quotients is roughly $m$ and the biggest one is at least $m$.
\end{cor}

\begin{lem}\label{rssr}
In the case of consecutive connections the role of $r$, $s$  in Theorem \ref{izrek} is interchanged
(see Figure \ref{fig2}).
\end{lem}

\begin{proof}
From assumptions we have
\begin{eqnarray*}
{p_{n-1}\over q_{n-1}}{P_{k-1}\over Q_{k-1}} =m= r s & \qquad
r={p_{n-1}\over Q_{k-1}} & \quad s={P_{k-1}\over q_{n-1}}\\
{p_{n-2}\over q_{n-2}}{P_{k-2}\over Q_{k-2}} =m=r_1 s_1 & \quad
 r_1={p_{n-2}\over Q_{k-2}} & \quad s_1={P_{k-2}\over q_{n-2}}
\end{eqnarray*}
We want to prove that $r_1=s$ and $s_1=r$.

\noindent
From equation (\ref{pqpq})
$$p_{n-1} q_{n-2}-p_{n-2}q_{n-1} =(-1)^n$$
we get
$$\frac{r Q_{k-1} P_{k-2}}{s_1} - \frac{r_1 Q_{k-2} P_{k-1}}{s}=(-1)^n$$
$$ Q_{k-1} P_{k-2} - Q_{k-2} P_{k-1} = (-1)^n \frac{ s s_1}{m}.$$

\noindent
From equation (\ref{pqpq}) and Lemma \ref{parity} it follows $s s_1= m$.
\end{proof}

\begin{lem}
If there are three consecutive connections between  $b_{n-1}$, $b_n$, $b_{n+1}$ and  $B_{k-1}$, $B_k$, $B_{k+1}$,  then  $rb_n-sB_k=0$ for the middle one connection (see Figure \ref{fig1} and \ref{fig2}).
\end{lem}

\begin{proof}
From assumptions we get using notation (\ref{rs}) from Theorem \ref{izrek} for three consecutive indexes
$$
r_{-1}={p_{n-2}\over Q_{k-2}},\,s_{-1}={P_{k-2}\over q_{n-2}},\,
r={p_{n-1}\over Q_{k-1}},\, s={P_{k-1}\over q_{n-1}},\,
r_1={p_{n}\over Q_{k}},\, s_1={P_{k}\over q_{n}}.
$$
From Lemma \ref{rssr} it follows
$$ r_{-1}=s=r_1, \qquad s_{-1}=r=s_1.$$
We get
\begin{equation}\label{alfa}
r\xi_n-s\eta_k = 0
\end{equation}
using
$$
sQ_{k-2}-p_{n-2} = r_{-1} Q_{k-2}-p_{n-2} = 0,  \;
P_{k-2}-rq_{n-2} = P_{k-2}-s_{-1} q_{n-2} = 0
$$
and (\ref{dd}).
Similarly we have
\begin{equation}\label{beta}
r_1 \xi_{n+1}-s_1 \eta_{k+1} = 0.
\end{equation}
In (\ref{beta}) we use definition of complete quotients (\ref{ksi}) and get
$$
\frac{s}{\xi_n - b_n} - \frac{r}{\eta_k - B_k}=0.
$$
After multiplying with denominators and using (\ref{alfa})
we get the result.
\end{proof}

\begin{rem}
Let us take the ladder of $(\sqrt[3]{2},\sqrt[3]{4})$ with length 1000.
In Figure \ref{fig3} we can see the difference $n-k$ for positions of 665 ladder connections.

\begin{figure}[!htbp]
\centering
\includegraphics[width=80mm]{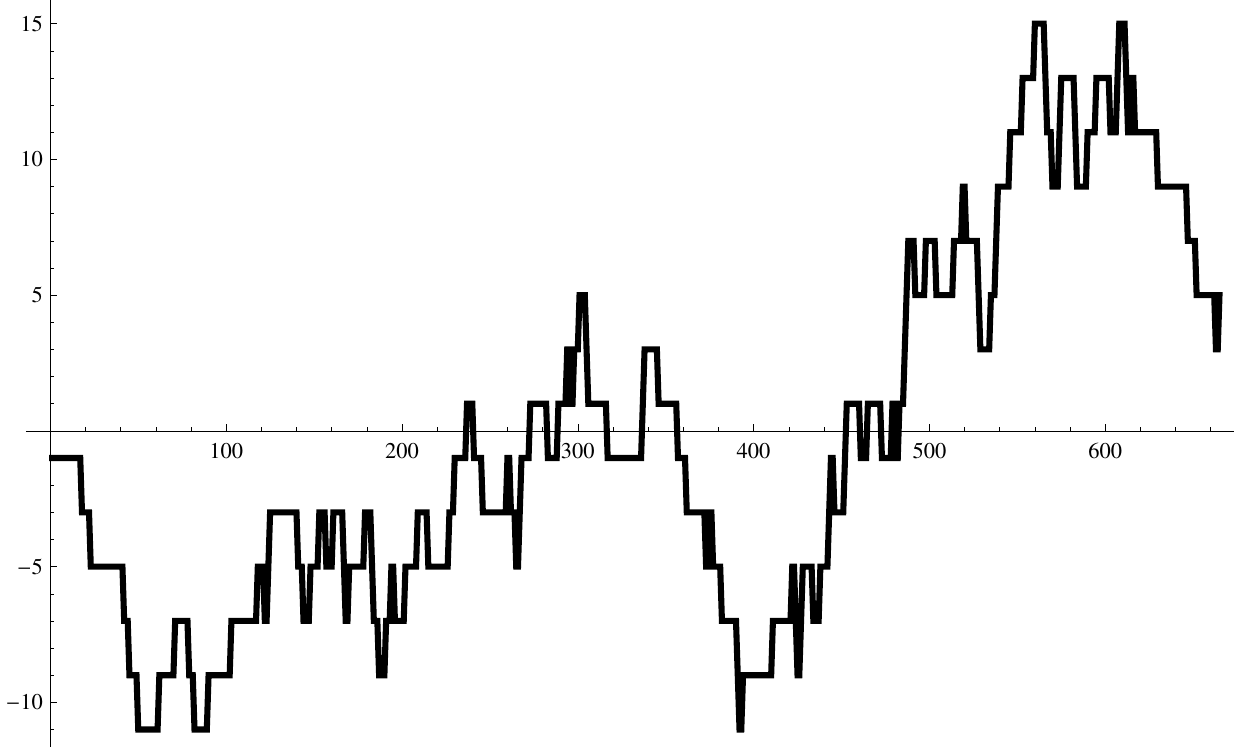}
\caption{$n-k$ for ladder of $(\sqrt[3]{2},\sqrt[3]{4})$ with  length 1000}\label{fig3}
\end{figure}
\end{rem}

\end{document}